\documentclass[a4paper, 12pt]{article}

\usepackage{amssymb}
\usepackage{amsmath}
\usepackage{amsthm}
\usepackage{url}
\usepackage{mathrsfs}
\usepackage{graphicx}
\usepackage{sectsty}
\usepackage[letterpaper,hmargin=1.0in,vmargin=1.0in]{geometry}
\usepackage[colorlinks,linkcolor=black,citecolor=black,filecolor=black,urlcolor=black]{hyperref}
\usepackage{color}

\usepackage[british]{babel}

\usepackage{xspace}
\usepackage{calc}
\usepackage{euscript}
\usepackage{amsfonts}
\usepackage{enumerate}

\usepackage{a4wide}
\parskip \medskipamount

\sectionfont{\large}
\subsectionfont{\normalsize}
\numberwithin{equation}{section}

\newtheorem{theorem}{Theorem} \newtheorem{lemma}{Lemma}[section]
\newtheorem{propo}{Proposition}[section]
\newtheorem{fact}{Fact}[section]

\newtheorem{question}{Question}[section]\newtheorem{example}{Example}[section]
 
\newtheorem{remark}{Remark}[section]

\def\P{\mathbb{P}}

\newcommand{\Ind}[1]{\mathbf{1}_{  \{#1\} }}

\newcommand{ \hit}{ t_{\mathrm{hit}} }
\newcommand{\E}{{\mathbb{E}}}

\renewcommand{\Pr}{ \mathrm P}

\newcommand{ \rel}{ t_{\mathrm{rel}} }
\newcommand{ \mix}{ t_{\mathrm{mix}} }
\newcommand{ \mixin}{ t_{\mathrm{mix}}^{(\infty)} }

\newcommand{ \TV}{ \mathrm{TV} }

\newcommand{\gd}{\delta}
\newcommand{\eps}{\epsilon}

\newcommand{\la}{\lambda}

\DeclareMathSymbol{\leqslant}{\mathalpha}{AMSa}{"36} 
\DeclareMathSymbol{\geqslant}{\mathalpha}{AMSa}{"3E} 
\DeclareMathSymbol{\eset}{\mathalpha}{AMSb}{"3F}     
\renewcommand{\le}{\;\leqslant\;}                   
\renewcommand{\ge}{\;\geqslant\;}                   

\renewcommand{\epsilon}{\varepsilon}

\newcommand{\sfrac}[2]{\mbox{\small $\frac{#1}{#2}$}}
\newcommand{\ssfrac}[2]{\mbox{\footnotesize $\frac{#1}{#2}$}}
\newcommand{\half}{\ssfrac{1}{2}}

\newcommand{\N}{\mathbb N}
\newcommand{\R}{\mathbb R}
\newcommand{\Z}{\mathbb Z}

\begin{document}

\title{A spectral characterization for concentration of the cover time}
\author{ Jonathan Hermon
\thanks{
University of Cambridge, Cambridge, UK. E-mail: {\tt jonathan.hermon@statslab.cam.ac.uk}. Financial support by
the EPSRC grant EP/L018896/1.}}
\date{}
\maketitle

\begin{abstract}
We prove that for a sequence of finite vertex-transitive graphs of increasing sizes, the cover times are asymptotically concentrated  iff the product of the spectral-gap and the expected cover time diverges. In fact, we prove this for general reversible Markov chains under the much weaker assumption (than transitivity) that the maximal hitting time of a state is of the same order as the average hitting time.

\end{abstract}


\tableofcontents

\section{Introduction}

A big part of the modern theory of  Markov chains  is dedicated to the study of the hierarchy of different quantities associated with a Markov chain. It is a common theme that certain phenomena can be characterized by a simple criterion concerning whether or not two such quantities are strongly separated (i.e., are of strictly different orders). Often, one of these quantities is the inverse of the spectral-gap. One instance is the cutoff phenomenon and the condition that the product of the mixing-time and the spectral-gap diverges, known as the \emph{product condition} (a necessary condition for precutoff in total-variation \cite[Proposition 18.4]{LPW} and  a necessary and sufficient condition for cutoff in $L_2$ \cite{Chen}). The condition that the product of the spectral-gap and the maximal (expected)  hitting time diverges is studied in \cite{aldoushitting}  and \cite[Theorem 1]{hermoninter}).    

Aldous'  classic criterion for  concentration of the   cover time \cite{aldouscover} is another such instance. Aldous' criterion asserts that for a sequence of  Markov chains on finite state spaces of diverging sizes  $\sfrac{ \tau_{\mathrm{cov}}^{(n)}}{ t_{\mathrm{cov}}^{(n)}} \to 1  $ in distribution if $t_{\mathrm{cov}}^{(n)}/\hit^{(n)} \to \infty $, where throughout the superscript `$(n)$' indicates that we are considering the $n$th Markov chain in the sequence, and where  $\tau_{\mathrm{cov}}=\inf\{t:\{X_s:s \le t \}=V \} $ is the \emph{cover time} of a Markov chain $(X_t )_{t \ge 0} $ on a finite state space $V$, defined to be the first time by which every state was visited at least once by the chain,    \[t_{\mathrm{cov}}:=\max_{x \in V  }\E_x[\tau_{\mathrm{cov}}]\]  its worst-case expectation and \[\hit:=\max_{x,y}\E_x[T_y], \qquad \text{where} \qquad T_y:=\inf\{t:X_t=y \}, \] is the maximal expected \emph{hitting time} of a state. More precisely,  Aldous \cite[Proposition 1]{aldouscover} showed that in the reversible case if $\hit^{(n)} =\Omega( t_{\mathrm{cov}}^{(n)})$ then there exists a sequence of initial states such that $\tau_{\mathrm{cov}}^{(n)}/t_{\mathrm{cov}}^{(n)} $ does not concentrate around any value.\footnote{The proposition is phrased for simple random walk, but the proof works for general reversible Markov chains. For a non-reversible counter-example consider a walk on the cycle with a fixed clockwise bias.}\footnote{In fact, by \cite[p.\ 274]{LPW}  starting from the stationary distribution $\pi$ one has that $\P_{\pi}[T_y>\hit/4] \ge 1/4$ for some state $y$. By \eqref{e:poscorfortail} $\P_{\pi}[T_y>i\hit] \ge(\P_{\pi}[T_y>\hit/4])^{4i} \ge 2^{-8i} $.}  Conversely, (even without reversibility) $t_{\mathrm{cov}}^{(n)} \le \min_y \E_{y}[\tau_{\mathrm{cov}}^{(n)}]+\hit^{(n)} $  and when   $t_{\mathrm{cov}}^{(n)}/\hit^{(n)} \to \infty $ we have that $\sfrac{t_{\mathrm{cov}}^{(n)} }{ \min_y \E_{y}[\tau_{\mathrm{cov}}^{(n)}]} \to 1 $ and that  $\sfrac{ \tau_{\mathrm{cov}}^{(n)}}{ t_{\mathrm{cov}}^{(n)}} \to 1  $ in distribution  for every sequence of initial states. 

Our Theorem \ref{prop:cover} refines Aldous' criterion in the transitive setup by allowing one to replace the maximal hitting time in his result by the inverse of the spectral-gap, which is positioned much lower in the aforementioned hierarchy of Markov chain parameters (see \eqref{e:mixlehit}). Throughout, let  $\mathrm{gap}:=\la_2 $ be the \textit{spectral-gap} of the considered chain, and  $\rel:=\frac{1}{\la_2}$ its \emph{relaxation-time}, where   $0=\la_1<\la_2 \le \ldots \la_{|V|} \le 2 $ are the eigenvalues of the Laplacian $I-P$. When considering simple random walk (\emph{SRW}) on a graph $G$ we often add parenthesis '$(G)$' to various quantities.\begin{theorem}
\label{prop:cover}
Let $G_{n}$ be a sequence of  finite connected vertex-transitive graphs of diverging sizes. Then $\sfrac{ \tau_{\mathrm{cov}}(G_n)}{ t_{\mathrm{cov}}(G_n)} \to 1  $ in distribution iff $\mathrm{gap}(G_n)t_{\mathrm{cov}}(G_n) \to \infty $.
\end{theorem}
We note that in the setup of Theorem \ref{prop:cover}, if $\mathrm{gap}(G_n)t_{\mathrm{cov}}(G_n) =O(1)  $ then  $\tau_{\mathrm{cov}}^{(n)}/t_{\mathrm{cov}}^{(n)} $ does not concentrate around any fixed value (by transitivity this holds for all initial states).\footnote{In fact, it is shown in \cite{basu} that for reversible chains there is always some state $x $ and a set $A$ of stationary probability at least $1/2$ such that $\Pr_x[T_A>t] \ge \exp(-\mathrm{gap} \, t) $ for all $t \ge 0$, where $T_A:=\inf\{t:X_t \in A\} $.} 
 
Theorem \ref{prop:cover} holds in the more general setup of reversible \emph{transitive Markov chains}. That is, reversible Markov chains on a finite state space $V$ whose transition matrix satisfies that for every $x,y \in V$ there is a bijection $f:V \to V$ such that $f(x)=y$ and $P(x,z)=P(y,f(z))$ for all $z \in V$.   Theorem \ref{prop:cover2} extends Theorem \ref{prop:cover}  to a much larger class of Markov chains. Denote the \emph{average hitting time} of an irreducible Markov chain on a finite state space $V$ by \[\alpha:=\sum_{x,y \in V }\pi(x) \pi(y) \E_x[T_y]=\sum_{y \in V } \pi(y) \alpha_y, \qquad  \] where throughout $\pi$ denotes the stationary distribution, and $\alpha_y:=\E_{\pi}[T_y]$. Theorem \ref{prop:cover2} indeed generalizes Theorem \ref{prop:cover}, as
 (by Fact \ref{f:a})  for a transitive chain $\alpha \le \hit \le 2 \alpha $ .  \begin{theorem}
\label{prop:cover2}
Consider a sequence of irreducible reversible Markov chains with finite state spaces $V^{(n)}$ and stationary distributions $\pi^{(n)}$. If\footnote{We write $o(1)$ for terms which vanish as $n \to \infty$. We write $f_n=o(g_n)$ or $f_n \ll g_n$ if $f_n/g_n=o(1)$. We write $f_n=O(g_n)$ and $f_n \lesssim g_n $ (and also $g_n=\Omega(f_n)$ and $g_n \gtrsim f_n$) if there exists a constant $C>0$ such that $|f_n| \le C |g_n|$ for all $n$. We write  $f_n=\Theta(g_n)$ or $f_n \asymp g_n$ if  $f_n=O(g_n)$ and  $g_n=O(f_n)$.} $\alpha^{(n)} \asymp \hit^{(n)} $, we have that $\sfrac{ \tau_{\mathrm{cov}}^{(n)}}{ t_{\mathrm{cov}}^{(n)}} \to 1$ in distribution for every sequence of initial states if and only if $\mathrm{gap}^{(n)}t_{\mathrm{cov}}^{(n)} \to \infty $. Moreover, if  $\min_{x \in V^{(n)} }\alpha_x \asymp t_{\mathrm{cov}}^{(n)} $ then 
\begin{equation}
\label{e:nonconcenntration}
\min_{v \in V^{(n)}  }\max_{x\in V^{(n)}} \inf_{C \ge 1}- \frac{1}{C} \log \Pr_v[T_x>Ct_{\mathrm{cov}}^{(n)}  ] \lesssim 1, 
 \end{equation}
 and   $\tau_{\mathrm{cov}}^{(n)}/t_{\mathrm{cov}}^{(n)} $ does not concentrate around any fixed value for any sequence of initial states.
\end{theorem}
As (\eqref{e:RT}) $\sfrac{1}{\mathrm{gap}} \le \alpha \le \hit \le t_{\mathrm{cov}} $,  by Theorem \ref{prop:cover2}    $\sfrac{1}{\mathrm{gap}^{(n)}} \asymp t_{\mathrm{cov}}^{(n)}$ if and only if $\alpha^{(n)} \asymp t_{\mathrm{cov}}^{(n)} $.
\begin{remark}
The condition  $\min_{x \in V^{(n)} }\alpha_x \asymp t_{\mathrm{cov}}^{(n)} $ above \eqref{e:nonconcenntration} is implied by  $\min_{x \in V^{(n)} }\alpha_x \asymp t_{\mathrm{hit}}^{(n)} $ in conjunction with $\mathrm{gap}^{(n)}t_{\mathrm{cov}}^{(n)} \asymp 1 $ (see \eqref{e:mixlehit}). Of course,  $\min_{x \in V^{(n)} }\alpha_x \asymp t_{\mathrm{hit}}^{(n)} $  is stronger than the condition $\alpha^{(n)} \asymp \hit^{(n)}$ used in the first part of the statement of Theorem \ref{prop:cover2}. We believe that the condition   $\min_{x \in V^{(n)} }\alpha_x \asymp t_{\mathrm{cov}}^{(n)} $ in \eqref{e:nonconcenntration} can be relaxed to   $\alpha^{(n)} \asymp t_{\mathrm{cov}}^{(n)} $.   
\end{remark}
Throughout we work with the continuous-time rate 1 version of the chain. We remark that all our  results are valid also in discrete-time even if the chain is not lazy (i.e., if $P(x,x)=0$ for some $x$). Moreover, in  this case one \emph{does not} need to replace $\mathrm{gap}$ by the absolute spectral-gap (as is often the case when translating a result from the continuous-time or discrete-time lazy setups to the discrete-time non-lazy setup). This can be verified by an application of Wald's equation (used to argue that the expected cover-time and hitting times are the same in both setups), together with the fact that Aldous' result \cite{aldouscover} (which is used in the proofs of Theorems \ref{p:transitivecover} and \ref{p:transitivecover2}) applies to both setups.

 We note that if $G_n$ and  $\widehat G_n $  are two  sequences of finite connected graphs of  uniformly bounded degree which are uniformly quasi isometric (i.e., there exists some $K>0$ such that $G_n$ is $K$-quasi isometric to $\widehat G_n$ for all $n$) then $\mathrm{gap}(G_n) \asymp \mathrm{gap}(\widehat G_n)$, $\alpha(G_n ) \asymp \alpha( \widehat G_n )$,  $\hit(G_n) \asymp \hit( \widehat G_n)$ and \cite[Theorem 1.6]{DLP} $t_{\mathrm{cov}}(G_n) \asymp t_{\mathrm{cov}}(\widehat G_n)$.\footnote{The fact that $\alpha(G_n ) \asymp \alpha( \widehat G_n )   $ follows from \eqref{e:RT} via a standard comparison argument \cite{DS}. The claim that  $\hit(G_n) \asymp \hit( \widehat G_n)$ can be seen from the commute-time identity (e.g.\ \cite[Eq.\ (10.14)]{LPW})   combined with the robustness of the effective-resistance under quasi isometries (cf.\ the proof of Theorem 2.17 in \cite{LP}).} In particular, if $\widehat G_n $ are vertex-transitive, the  sequence of SRWs on $G_n$ satisfies the conditions of Theorem \ref{prop:cover2} (apart perhaps from   $\min_{x \in V^{(n)} }\alpha_x \asymp t_{\mathrm{cov}}(G_n)$, used only in \eqref{e:nonconcenntration}).

The cover time of an $n \times n$ grid torus is concentrated \cite{DPRZ}, while that of the $n$-cycle is not. The following example shows that an $n \times \lceil n/\log^2 n \rceil $ grid torus is in some sense critical.

\begin{example} Consider an $n \times m $ discrete (grid) torus (i.e., the Cayley graph of $\Z_n \times \Z_{m} $ w.r.t.\ the standard choice of generators). If $m=m(n)=O( n/\log^2 n)$ then its (expected) cover time
 is of order $n^2$, same as the inverse of its spectral-gap. Conversely, if $ n/\log^2 n  \ll m\le n$   the  cover time
 is of order $mn (\log n)^{2} \gg n^2 $, while the  spectral-gap is $\Theta(\sfrac{1}{n^2}) $. 
\end{example}



Theorem \ref{prop:cover} is a fairly immediate consequence of the following result (see \S\ref{s:aux} for the details). 
\begin{propo}
\label{p:transitivecover}
For every transitive Markov chain on a finite state space $V$,   
\begin{equation}
\label{e:M>5}
t_{\mathrm{cov}} \ge \sfrac{1}{4} \hit \log \left( \hit \, \mathrm{gap}    /12\right).
\end{equation}
\end{propo}
We note that whenever $\sfrac{1}{ \mathrm{gap}} \lesssim   |V|^{a} $ for some $a \in (0,1)$ the bound offered by \eqref{e:M>5} is of the correct order, as by Matthews \cite{Matthews} $t_{\mathrm{cov}} \le \hit \sum_{i=1}^{|V|-1}\sfrac{1}{i} \le \hit (\log |V|+1)  $ (e.g., \cite[Theorem 11.2]{LPW}) and by \eqref{e:RT}  $\hit \ge \alpha \ge \sfrac{|V|}{2} $ (and so $\mathrm{gap} \, \hit \gtrsim |V|^{1-a}$).   

Theorem \ref{prop:cover2} (apart from \eqref{e:nonconcenntration}) is a fairly immediate consequence of the following  extension  of Proposition \ref{p:transitivecover}.

\begin{propo}
\label{p:transitivecover2}
For every irreducible reversible Markov chain on a finite state space $V$ with a stationary distribution $\pi$ we have that\begin{equation}
\label{e:M>99}
\begin{split}
 t_{\mathrm{cov}} \ge \frac{ \alpha}{4}  \log \left( M \left(\hit/\alpha ,\alpha \, \mathrm{gap}\right)  \right), \quad \text{where} \quad M(a,b)=\sfrac{ 8\log \left( 8a \right)+b }{  32a\log \left( 8a \right)} 
\end{split}
\end{equation}
\end{propo}
When $\alpha \asymp \hit $ and $\hit \, \gg 1/\mathrm{gap}  $ we get that $M(\sfrac{\hit}{\alpha} ,\alpha \, \mathrm{gap}) \gg 1$ and so $t_{\mathrm{cov}} \gg \hit  $.

\subsection{Related work}
Cover times have been studied extensively for over 35 years. This is a  topic with rich ties to other objects such as the Gaussian free field \cite{DLP}. There has been many works providing general bounds on the cover time, studying its evolution  and its fluctuations  in general \cite{Zhai,Ding,MP,Bar,Kahn}, and in particular for the giant component of various random graphs \cite{CF}, for trees \cite{aldousr,DZ},  for the two dimensional torus \cite{DPRZ,Ding2,CP,BK} and for higher dimensional tori \cite{Bel}. Feige \cite{Feige1,Feige2} proved tight extremal upper and lower bounds on cover times of graphs (by SRW). For a more comprehensive review of the literature see the related work section in \cite{DLP} and the references therein.  For further background on  hitting times and cover times see \cite{aldousfill,LPW,aldoushitting}.

\subsection{Organization of this note}
In \S\ref{s:open} we present some open problems. In \S\ref{s:def} we present some background on hitting-times. In \S\ref{s:aux} we  prove Theorems \ref{prop:cover}-\ref{prop:cover2} and  Propositions \ref{p:transitivecover} and \ref{p:transitivecover2}.  Example 1.1 is analyzed in \S\ref{s:EX3}. 

\section{Open Problems}
\label{s:open}
In the following seven questions let $G_n=(V_n,E_n) $ be a sequence of finite connected vertex-transitive graphs of diverging sizes. Denote the degree of $G_n$ by $d_n$ and its diameter (i.e.\ the maximal graph distance between a pair of vertices) by   $\mathrm{Diam}(G_n)$. Let $\rel(G_{n}):=\sfrac{1}{\mathrm{gap}(G_n)}$ be the \emph{relaxation-time}. The following two questions are the focus of an ongoing work  with Nathana\"el Berestycki  (we believe both have an affirmative answer).
 \begin{question}
\label{Q:2}
 Assume that    $\rel(G_n) \ll \hit(G_n) $. Is it the case that $\sfrac{t_{\mathrm{cov}}(G_n)}{\hit(G_n) \log |V_n|} \to 1 $?
\end{question}
 \begin{question}
\label{Q:2'}
 Assume that    $\mathrm{Diam}(G_n)^2 \ll |V_{n}| $ and $d_n \asymp 1$. Does $\sfrac{t_{\mathrm{cov}}(G_n)}{\hit(G_n) \log |V_n|} \to 1 $?
\end{question}
We now discuss some relaxations of the conditions from Questions \ref{Q:2} and \ref{Q:2'}.  
Let $\mathcal{R}(a,b)  $ be the \emph{effective resistance} between $a$ and $b$ and $\mathcal{R}_*:=\max_{x,y } \mathcal{R}(x,y)  $ (e.g.\ \cite[Ch.\ 9]{LPW} and \cite[Ch.\ 2]{LP}). Let  $o_{n} \in V_n$.  Consider the conditions:
\begin{itemize}
\item[(i)] $\mathcal{R}_* (G_{n}  )\asymp 1/d_n $ (where $\mathcal{R}_*(G)$ is $\mathcal{R}_*$ for SRW on $G$),   
\item[(ii)]  $|\{x \in V_n : \mathcal{R}(o_{n},x)  \le \delta\mathcal{R}_*(G_{n}  ) \}|=O(1) $ for all fixed $\delta \in (0,1) $. 
\end{itemize}
 \begin{question}
\label{Q:3} 
Does    $\rel(G_n) \log |V_n| \lesssim \hit(G_n) $ imply (i)? 
\end{question}
 \begin{question}
\label{Q:3'} 
Does    $\rel(G_n) \log |V_n| \ll \hit(G_n) $ imply (ii) when $d_n \asymp 1 $? 
\end{question}
 Condition (i) arises in forthcoming work of Tessera and Tointon \cite{TT} as the analogue of transience in a Varopoulos-type result for finite vertex-transitive graphs. In particular, they show (in the transitive setup) that it follows from the condition $\mathrm{Diam}(G_n)^2 \lesssim \frac{|V_n|}{\log |V_n|}$.  Below we give some equivalent conditions to (i). One of them is $\hit(G_n) \asymp |V_n|$. Using the fact that $\rel(G) \le 2 d \mathrm{Diam}(G)^2  $  for a vertex-transitive graph $G$ of degree $d$ \cite[Theorem 13.26]{LPW} (as well as  $\hit(G_n) \gtrsim |V_n|$), it follows that when $d_n \asymp 1 $,   
\begin{equation}
\label{e:TTimplications}
\begin{split}
& \mathrm{Diam}(G_n)^2 \log |V_n| \lesssim |V_n| \implies \rel(G_n) \log |V_n| \lesssim \hit(G_n), \text{ and likewise} 
\\ & \mathrm{Diam}(G_n)^2 \log |V_n|\ll  |V_n| \implies \rel(G_n) \log |V_n| \ll \hit(G_n) .
\end{split}
\end{equation} 

Moderate growth is a certain technical growth condition introduced by Diaconis and Saloff-Coste in their seminal work \cite{DSC}. For Cayley graphs this condition is shown by Breuillard and Tointon \cite{BT} to be equivalent  to the condition   $c|V| \le \mathrm{Diam}(G)^a $, in some precise quantitative sense, with these $a$ and $c$ being related to the parameters in the definition in \cite{DSC}. This was recently extended to vertex-transitive graphs by Tessera and Tointon \cite{TT2}. 
  
Using the fact that for vertex-transitive graphs of moderate growth  $\rel(G_n) \gtrsim \mathrm{Diam}(G_n)^2  $ \cite{DSC},\footnote{See also \cite[\S8.1]{Ex}, where it is noted that argument from \cite{DSC} is valid for vertex-transitive graphs of moderate growth, not just for Cayley graphs.} it appears that the main ingredient for establishing the converse implications to the ones in \eqref{e:TTimplications} when $d_n \asymp 1$  is providing an affirmative answer to Question \ref{Q:3}. Indeed for graphs of sufficiently large growth  the condition  
  $\mathrm{Diam}(G_n)^2 \ll \frac{|V_n|}{\log |V_n|}$ holds for free.  
 \begin{question}
\label{Q:newer}
Do the reverse implications in \eqref{e:TTimplications} hold? What can be said without the assumption that $d_n \asymp 1 $? 
\end{question}  
   \begin{question}
\label{Q:new}
    Assume that $\mathrm{Diam}(G_n)^2 \ll \frac{|V_n|}{\log |V_n|} $. Is it the case that for all fixed $\delta \in (0,1) $,   $|\{x \in V_n : \mathcal{R}(o_{n},x)  \le \delta\mathcal{R}_*(G_{n}  ) \}|=O(d_{n})$?
\end{question}
For vertex-transitive graphs condition (i) is equivalent to   $\hit(G_n) \asymp |V_n| $  and condition (ii)  is equivalent to the condition that $|\{x \in V_n : \E_{o_n}[T_x] \le \delta \hit(G_{n}) \}|=O(1) $ for all fixed $\delta \in (0,1) $. Indeed, by the commute-time identity (e.g., \cite[Proposition 10.7]{LPW}), for SRW on a  graph $G=(V,E)$ we have that   \begin{equation}
 \label{e:commutetimeid}
 \half (\E_{o}[T_x]+\E_{x}[T_o])=|E|\mathcal{R}(o,x).   
 \end{equation}
If $G=(V,E) $ is also vertex-transitive, then $\E_{o}[T_x]=\E_{x}[T_o]$ \cite[Proposition 2]{aldoushitting} (see also \cite[Proposition 10.10]{LPW}, this also follows from \eqref{e:Z}), and so $\E_{o}[T_x]=|E|\mathcal{R}(o,x) $. 
Hence  
   \[\mathcal{B}_{\mathrm{eff-res}}(o,\delta\mathcal{R}_*(G)):=\{x \in V  : \mathcal{R}(o,x)  \le \delta\mathcal{R}_* (G)\} =\{x : \E_{o}[T_x] \le \delta \hit(G) \}, \]
and (even without transitivity)    $\hit(G_n) \asymp |V_n| $ iff  $\mathcal{R}_* (G_{n}  )\asymp 1/d_{n} $.
\begin{question}
In the above setup, is it the case that   $\hit(G_n) \asymp |V_n| $ if and only if $t_{\mathrm{cov}}(G_{n})\asymp |V_n| \log |V_n|$?
\end{question}
The implication $\hit(G_n) \asymp |V_n| \implies t_{\mathrm{cov}}(G) \asymp  |V_n| \log |V_n|  $ follows (even without transitivity) from \eqref{e:mixlehit} and the bound $t_{\mathrm{cov}}(G)\ge (1-o(1))|V| \log |V| $  \cite{Feige1} holding for every $G=(V,E)$. 

It is not clear what is the correct analog of transience for a sequence of  finite graphs. One technical definition   is given in \cite{MP}. Here we discuss two different conditions. Our goal is to motivate Questions \ref{Q:3}-\ref{Q:newer}, relate them back to $\tau_{\mathrm{cov}}$  and stimulate future research.  

A natural informal  definition of being \emph{uniformly locally transient} for a sequence $G_n=(V_n,E_n)$ of $d_n$-regular graphs is that the expected number of returns to the origin of the walk by the mixing-time (or by time $|V_n|$) is $O(1)$ (uniformly in the choice of the origin). If we think of ``mixing" as ``reaching infinity" then this is a natural analog of transience in the infinite setup. It is hence natural that such a condition can be phrased in terms of effective resistance. To make this precise,  one can consider the following equivalent conditions: \begin{itemize}
\item $\mathcal{R}_*(G_n) \asymp 1/d_n $ (as discussed above, this is equivalent to $\hit \asymp |V_n| $).
\item  $\min_{x,y \in V_n:\, x\neq y} \mathrm{P}_x[T_x^+>T_y] \gtrsim 1 $, where $T_x^{+}:=\inf\{t:X_t=x, X_s \neq x \text{ for some }s<t \}$.
\item $ \max_{x,y \in V_n } \int_{0 }^{\rel(G_n)} H_t^{G_{n}}(x,x)dt = O(1)$ (where $H_t^{G_{n}}(\cdot,\cdot \cdot)$ are the time $t$ transition probabilities for SRW on $G_n$; equivalently, $\max_{x\in V_n } \int_{0 }^{\rel(G_n)} H_t^{G_{n}}(x,y)dt = O(1)$). 
\item $ \max_{x,y \in V_n } \int_{0 }^{|V_n| \vee \mixin(G_n) } H_t^{G_{n}}(x,y)dt  = O(1)$, where $a \vee b:=\max\{a,b\}$ and $ \mixin:=\inf\{t:\max_{a,b }|\frac{H_t(a,b)}{\pi(b)}-1| \le 1/2 \} $ is the $L_{\infty}$  mixing time.  
\end{itemize}

Consider a sequence of finite connected graphs $G_n:=(V_n,E_n)$. Let  $d_{\mathrm{max}}^{(n)}$ and   $d_{\mathrm{min}}^{(n)}$  be the maximal and minimal (respectively) degree of $G_n$. We assume that $d_{\mathrm{max}}^{(n)} \asymp d_{\mathrm{min}}^{(n)}$ so that $\pi_{G_n}$ the stationary distribution of SRW on $G_n$ satisfies $\max_{v \in V_n }\pi_{G_n}(v) \asymp \sfrac{1}{|V_n|}\asymp \min_{v \in V_n }\pi_{G_n}(v)$.  A natural informal definition for saying that $G_n$ is \emph{uniformly globally transient} is that the walk either returns to the origin rapidly (i.e.\ in $O(1)$ time units), or else it is  unlikely to return before getting mixed. To make this precise, we consider the condition \[\text{condition (a):} \qquad \lim_{s \to \infty} \limsup_{n \to \infty} \max_{x,y \in V_n } \int_{s \wedge \rel(G_{n}) }^{\rel(G_n)} H_t^{G_{n}}(x,y)dt=0. \]
In fact, condition (a) is equivalent to (b):  $\limsup_{n \to \infty} \max_{x\in V_n } \int_{s \wedge \rel(G_{n}) }^{\rel(G_n)} H_t^{G_{n}}(x,x) dt \to 0 $, as $s \to \infty $, as well as to (c):   $\limsup_{n \to \infty} \max_{x\in V_n } \int_{s \wedge \mixin(G_{n}) }^{\mixin(G_n)} H_t^{G_{n}}(x,x) dt \to 0 $, as $s \to \infty $ (provided  $d_{\mathrm{max}}^{(n)} \asymp d_{\mathrm{min}}^{(n)}$). These conditions imply that    
\begin{itemize}
\item[(1)]  $\mathcal{R}_* (G_{n}  )\asymp 1/d_{\mathrm{max}}^{(n)} $ (recall that we assume   $d_{\mathrm{max}}^{(n)} \asymp d_{\mathrm{min}}^{(n)}$) , 
\item[(2)] $\max_{o \in V_n} |\{x \in V_n : \E_{o}[T_x] \le \E_{\pi_{G_{n}}}[T_y]- \delta |V_{n}|   \}| =O(1)$ for all fixed $\delta >0 $,
\item[(3)] $\rel(G_n)=o(\hit(G_n))$.
\end{itemize}
  When $G_n$ are vertex-transitive,  condition (1)-(3) are equivalent to condition (a) and also  to conditions  (1),(2') and (3), where condition (2') is:  $|\mathcal{B}_{\mathrm{eff-res}}(o,\delta\mathcal{R}_*(G_{n}))|=O(1)$, for all fixed $\delta \in (0,1)$ (this is  condition (ii) above).  While we omit the proof of these equivalence, we stress that some of them are not at all obvious. 

We strongly believe that (in the transitive setup) these conditions imply that $\sfrac{t_{\mathrm{cov}}(G_n)}{\hit(G_n) \log |V_n|} \to 1  $, and that the distribution of the cover time exhibits Gumbel fluctuations.  

Considering  an $n \times n$ discrete torus shows that even under transitivity the condition   $\rel(G_n) \log |V_n| \asymp \hit(G_n)$ does not imply that    $\mathcal{R}_* (G_{n}  )\asymp 1 $. Another illustrative example is an  $n \times n \times f(n) $ discrete torus. It is not hard to verify that when $f(n)=\lceil \log n \rceil$, we have that     $\rel(G_n) \log |V_n| \asymp \hit(G_n) $ and   $\mathcal{R}_* (G_{n}  )\asymp 1 $.  However for every $\delta \in (0,1)$ it holds that  $|\{x \in V_n : \mathcal{R}(o_{n},x)  \le \delta\mathcal{R}_*(G_{n}  ) \}|=\Omega_{\delta}(n^{ 2 c \delta} \log n )$ for some absolute constant $c \in (0,1) $, and so condition (ii) fails. Conversely, if $\log n \ll f(n) \lesssim n $ then      $\rel(G_n) \log |V_n| \ll \hit(G_n) $  and conditions (i) and (ii) hold.

Considering a cartesian product of the $n$-cycle with a vertex-transitive expander of size $f(n) \gg n $ shows that  
the condition     $\rel(G_n) \log |V_n| \ll \hit(G_n) $ above is not a necessary condition for conditions (i) and (ii) to hold.

In light of Example 1.1, the following  question naturally arises.
 \begin{question}
\label{Q:1.2}

Let $G_n$ be a sequence of finite connected vertex-transitive graphs of diverging sizes and uniformly bounded degrees. Assume that along every subsequence   $\sfrac{ \tau_{\mathrm{cov}}(G_n)}{ t_{\mathrm{cov}}(G_n)} $ does not converge to 1 in distribution.
Is it the case that  when viewing $G_n$ as a metric space with the graph distance as its metric,  after rescaling distances by a $\frac{\mathrm{Diam}(G_n)}{f (\mathrm{Diam}(G_n))}$ factor,
for every $f:\N \to \R_+$ satisfying $1 \ll f(k) = o((\log k)^{2})$,  the pointed  Gromov-Hausdorff  scaling limit exists and is $\R$?
\end{question}
This question is the cover time analog of a question from \cite{BH2}\footnote{In which the assumption on    $\sfrac{ \tau_{\mathrm{cov}}(G_n)}{ t_{\mathrm{cov}}(G_n)} $ is replaced by the assumption that $\mixin(G_n) \asymp \hit(G_n)$, and there we take $f(k) =o(\log k)$).}, where it is shown that for a sequence of finite vertex-transitive graphs $G_n$ of fixed degree and increasing sizes satisfying that their mixing times are proportional to their  maximal hitting times,  $G_n$ rescaled by their diameters converge in the Gromov-Hausdorff topology to  the unit circle $S^1$.

\section{Hitting-times preliminaries}
\label{s:def} 
 Let $(X_t)_{t=0}^{\infty}$ be an irreducible reversible  Markov chain on a finite state space $V$ with transition matrix $P$ and stationary distribution $\pi$. Denote the law of the continuous-time rate 1 version of the chain starting from vertex $x$ (resp.\ initial distribution $\mu$) by $\mathrm{P}_x$ (respectively, $\Pr_{\mu}$). Denote the corresponding expectation by $\mathbb{E}_x$ (respectively, $\E_{\mu}$). Let $H_t:=e^{-t(I-P)}$ be its heat kernel (so that $H_t(\cdot, \cdot)$ are the time $t$ transition probabilities). 
%
%
%

We now present some background on hitting times. The random target identity (e.g.~\cite[Lemma 10.1]{LPW}) asserts that $ \sum_y \pi(y) \mathbb{E}_x[T_y] $ is independent of $x$ and hence equals $\alpha$, while for all $x \in V$ we have that (e.g.~\cite[Proposition 10.26]{LPW}) 
\begin{equation}
\label{e:alphax}
 \alpha_x:=  \E_{\pi}[T_x]=\frac{1}{\pi(x)} \sum_{i=0}^{\infty}\left( P^{i}(x,x)-\pi(x) \right)=\frac{1}{\pi(x)} \int_{0}^{\infty}\left( H_t(x,x)-\pi(x) \right)dt, \end{equation}

Averaging over $x$ yields the eigentime identity (\cite[Proposition 3.13]{aldousfill})  
\begin{equation}
\label{e:RT}
\alpha=\sum_{x,y} \pi(x)\pi(y) \mathbb{E}_x[T_y]=\sum_{y} \sum_{i=0}^{\infty}(P^{i}(y,y)-\pi(y))=\sum_{i=0}^{\infty}\mathrm{[Trace}(P^i)-1]=\sum_{i \ge 2}\frac{1}{\la_i}. \end{equation}
For a transitive Markov chain  $ \mathbb{E}_x[T_y]= \mathbb{E}_y[T_x] $ for all $x,y$ \cite[Proposition 2]{aldoushitting} (see also \eqref{e:ZZZ}) and so \[\forall \, x, \quad \alpha=\E_{\pi}[T_x].\]
Let $U \sim \pi $ be independent of the chain. As $T_x \le T_U+\inf\{t:X_{t+T_U}=x \}$,  using the random target identity to argue $\E[T_{U}]=\alpha $, as well as the strong Markov property yields:
\begin{fact}[\cite{LPW} Lemma 10.2]
\label{f:a}
  $\max_x \alpha_x  \le \hit \le \alpha+\max_x \alpha_x \le  2\max_x \alpha_x $.
\end{fact}
The following material can be found at \cite[\S 3.5]{aldousfill}.  Under reversibility, for any set $A$ the law of its hitting time $T_A:=\inf\{t:X_t \in A\} $ under initial distribution $\pi$ conditioned on $A^{\complement} $, is a mixture of Exponential distributions, whose minimal parameter $\lambda(A)$ is the Dirichlet eigenvalue of the set $A^{\complement}$. There exists a distribution $\mu_A$, known as the \emph{quasi-stationary distribution} of $A^{\complement}$, under which  $T_A$ has an Exponential distribution of parameter $\la(A)$.  It follows that $\la(A) \ge \frac{1}{\max_a \E_a[T_A]} $. We see that for all $t \ge 0$,  \begin{equation}
\label{e:exptailpi} \Pr_{\pi}[T_y> t] \le \exp(-t/\hit) , \quad \text{and so} \quad \E_{\pi}[T_y^2] \le 2 \hit^{2}. 
\end{equation} 
Using the above description of the law of $T_A$ it is not hard to show \cite[p.\ 86]{aldousfill} that
\begin{equation}
\label{e:poscorfortail} \forall \, s,t \ge 0, \qquad \Pr_{\pi}[T_y> t+s \mid T_y \ge s ] \ge \Pr_{\pi}[T_y> t ]. 
\end{equation} 
Let
\[Z_{x,y}:=\int_0^{\infty}\left(H_t(x,y)-\pi(y)\right)dt. \]   
By \eqref{e:alphax} $\pi(y) \alpha_y:=Z_{y,y} $ for all $y$. We also have that for all $x,y \in V$ (e.g.\ \cite[\S2.2]{aldousfill})
\begin{equation}
\label{e:Z}
 \alpha_y -\E_x[T_y]=Z_{x,y}/\pi(y)\, \stackrel{\text{(under reversibility)}}{ =}Z_{y,x}/\pi(x)=\alpha_x- \E_y[T_x].
\end{equation}
Hence by \eqref{e:alphax} (this is justified in more detail below) 
\begin{equation}
\label{e:ZZZ}
\E_x[T_y]=\frac{1}{\pi(y)}\int_0^{\infty}\left(H_t(y,y)-H_{t}(x,y)\right)dt.
\end{equation} 
Indeed, it follows from the spectral decomposition (e.g., \cite[\S12.1]{LPW}) that for all $x$ and  all $s,t \ge 0$ we have that   
 \begin{equation}
 \label{e:specd}
0< H_{t+s}(x,x) -\pi(x) \le e^{-s/\rel}(H_{t}(x,x) -\pi(x)). 
\end{equation} By a standard application of reversibility and the Cauchy-Schwartz inequality (e.g.\  \cite[Eq.\ (3.59)]{aldousfill}) $|\sfrac{H_t(y,x)}{\pi(x)}-1|^{2}=|\sum_z \pi(z)\sfrac{H_{t/2}(y,z)- \pi(z)}{\pi(z)}\sfrac{H_{t/2}(x,z)- \pi(z) }{\pi(z)}|^{2} \le  \left(\sfrac{H_{t}(x,x)}{\pi(x)}-1\right)\left(\sfrac{H_{t}(y,y)}{\pi(y)}-1\right)   $ (where we have used reversibility to get $\sum_z \pi(z)\left( \sfrac{H_{s}(a,z)- \pi(z)}{\pi(z)}\right)^2=\sfrac{H_{2s}(a,a)}{\pi(a)}-1 $). Hence
\begin{equation}
\label{e:diaglarger}
|\sfrac{H_t(y,x)}{\pi(x)}-1| \le \half \left(\sfrac{H_{t}(x,x)}{\pi(x)}-1\right)+\half \left(\sfrac{H_{t}(y,y)}{\pi(y)}-1\right). \end{equation}
Combining \eqref{e:specd} and \eqref{e:diaglarger} we see that
 \[\max_{a,b} \int_{0}^{\infty}\frac{|H_{t}(a,b)-\pi(b) |}{\pi(b)} dt  = \max_a \int_{0}^{\infty}\frac{H_{t}(a,a)-\pi(a) }{\pi(a)} dt \le \max_a \int_0^{\infty}\frac{e^{-t/\rel}}{\pi(a)}dt<\infty .\] Hence $\int_{0}^{\infty}(H_{t}(y,y)-\pi(y))  dt -\int_{0}^{\infty}(H_{t}(x,y)-\pi(y))  dt=\int_{0}^{\infty}\left(H_{t}(y,y)-H_{t}(x,y)\right) dt $, and so \eqref{e:ZZZ} indeed follows from \eqref{e:alphax} and \eqref{e:Z}. Combining \eqref{e:specd} and \eqref{e:diaglarger} also yields: 

\begin{lemma}
\label{lem:auxexpdecay} For every irreducible, reversible Markov chain on a finite state space $V$ with a stationary distribution $\pi$ we have that

\begin{equation}
\label{e:Z1}
\forall \, x,y \in V,s \ge 0, \qquad  \int_s^{\infty}\left(\frac{H_t(y,x)}{\pi(x)}-1\right)dt \le \half e^{-s/\rel}(\alpha_{x}+\alpha_y) .
\end{equation}
\end{lemma}
\emph{Proof.}
Using \eqref{e:diaglarger}, \eqref{e:specd} and \eqref{e:alphax} (in this order) 
\[   \int_s^{\infty} \left(\sfrac{H_t(y,x)}{\pi(x)}-1\right)dt  \le \half  \int_s^{\infty} \left(\sfrac{H_t(x,x)}{\pi(x)}-1\right)dt+\half  \int_s^{\infty} \left(\sfrac{H_t(y,y)}{\pi(y)}-1\right)dt  \]
\[\le \half e^{-s/\rel}  \int_0^{\infty} \left(\sfrac{H_t(x,x)}{\pi(x)}-1\right)dt+\half e^{-s/\rel}  \int_0^{\infty} \left(\sfrac{H_t(y,y)}{\pi(y)}-1\right)dt= \half e^{-s/\rel}(\alpha_{x}+\alpha_y). \, \qed \]
Recall that   $ \mix^{\mathrm{TV}}:=\inf\{t:\max_a \sum_b|H_t(a,b)-\pi(b) | \le 1/2 \}$ and  $ \mixin:=\inf\{t:\max_{a,b} |\frac{H_t(a,b)}{\pi(b)}-1 | \le 1/2 \} $ are the \emph{total-variation} and the $L_{\infty}$ \emph{mixing times}.
Under reversibility the  quantities considered above satisfy the following  
  hierarchy (e.g.\ \cite[Theorems 10.22 and  12.5]{LPW}): 
\begin{equation}
\label{e:mixlehit}
\sfrac{1}{9\la_2} \log 4 \le \sfrac{1}{9}\mix^{\TV}   \le \sfrac{1}{9} \mixin \le \hit  \le t_{\mathrm{cov}} \le \hit (\log |V| +1 ), 
\end{equation}
where the last inequality is due to Matthews' \cite{Matthews} (see \cite[Ch.\ 11]{LPW} for a neat presentation). It is interesting to note that for reversible chains  $ \mix^{\TV} \le C \min_x \max_y \E_y[T_x]$, for some absolute constant $C$. This follows from the results of Lov{\'a}sz
and Winkler \cite{LW} concerning what they call the ``forget-time".\footnote{Under reversibility, it follows from their result that $t_{\mathrm{stop}} \le C_1 t_{\mathrm{forget-time}} \le C_{1} \min_x \max_y \E_y[T_x] $, while Aldous \cite{aldous1982some} showed that $ \mix^{\TV} \le C_2 t_{\mathrm{stop}}$ (see also Peres and Sousi \cite{PS}), for some absolute constants $C_{1},C_2$, where $t_{\mathrm{stop}} $ is the expectation of a mean optimal stopping rule (\eqref{eq:deftstop}) starting from the worst initial state.} See equation (3.7) in \cite{HLP} for a more elementary derivation.

\section{Proof of Theorems    \ref{prop:cover} - \ref{prop:cover2} and Propositions \ref{p:transitivecover} and \ref{p:transitivecover2}}
\label{s:aux}
We will show that for every irreducible reversible Markov chain on a finite state space $V$,
\begin{equation}
\label{e:M>88}
\forall \, x \in V,\eps \in (0,1), \qquad \pi (B(x,\eps)) \le \sfrac{2  \log \left( \sfrac{2\hit}{ \eps \alpha } \right)}{2  \log \left( \sfrac{2\hit}{ \eps \alpha} \right)+\eps  \alpha \, \mathrm{gap}  },  
\end{equation}
where $B(x,\eps):=\{ z \in V : \mathbb{E}_z[T_x] \le \alpha_x - \eps \alpha \text{ or } \mathbb{E}_x[T_z] \le \alpha_z - \eps \alpha \} $, whereas if it is transitive,
 \begin{equation}
\label{e:M>4}
\forall \, x \in V, \eps \in (0,1), \qquad | \{ z:\mathbb{E}_z[T_x] \le (1-\eps)\alpha\}|\le \sfrac{2  \log (2/\eps) }{  2\log (2/\eps)+\eps\alpha \, \mathrm{gap}  } |V| .
\end{equation} 
We first prove \eqref{e:M>5} and \eqref{e:M>99} assuming \eqref{e:M>4} and \eqref{e:M>88}, whose proofs are deferred to the end of the section.

\emph{Proof of \eqref{e:M>5} and \eqref{e:M>99}.} We first prove \eqref{e:M>5}. By \eqref{e:RT}  $|V| \ge \alpha \, \mathrm{gap} $ (this is used in the first inequality below). By Fact \ref{f:a} $\alpha \ge \half \hit$. Using the fact that for transitive chains $\mathbb{E}_a[T_b]=\mathbb{E}_b[T_a]$ \cite{aldoushitting},  by \eqref{e:M>4} with $\eps=\half$ there exists a set $B\subseteq V$ of size at least \[\left\lceil |V|/\left(|V|\sfrac{4  \log 4 }{  4\log 4+\alpha \, \mathrm{gap}  } \right) \right\rceil \ge \frac{\alpha \, \mathrm{gap}}{4  \log 4} \ge \frac{\hit \, \mathrm{gap}}{8 \log 4 } \ge \frac{\hit \, \mathrm{gap}}{12},   \] such that for all $a,b \in B$ we have that $\mathbb{E}_a[T_b] \ge \alpha/2 \ge \hit/4 $. The claim now follows from Matthews' method \cite{Matthews}  (see \cite[Proposition 11.4]{LPW}), which asserts that $t_{\mathrm{cov}} \ge$ $ \min_{a,b \in B:\, a \neq b  }\mathbb{E}_a[T_b]\sum_{i=1}^{|B|-1}\sfrac{1}{i}$ (using $\sum_{i=1}^{k-1}\sfrac{1}{i} \ge \log k $).

We now prove \eqref{e:M>99}. We first use the Paley-Zygmund inequality to argue that
\[\pi(D) \ge \frac{\alpha^2}{4 \sum_z \pi(z) \alpha_z^2 } \ge \frac{\alpha}{4 \hit }, \qquad \text{where} \qquad D:=\{z:\alpha_z \ge \alpha/2 \}. \]
Hence by \eqref{e:M>88} with $\epsilon=\sfrac{1}{4}$  there exists a subset $B$ of $D$ of stationary probability at least  \[ \left\lceil \pi(D)/\left(  \sfrac{  8\log \left( \sfrac{8\hit}{  \alpha} \right)}{ 8\log \left( \sfrac{8\hit}{  \alpha} \right)+\mathrm{gap} \cdot \alpha}  \right) \right\rceil ,  \] such that for all $a,b \in B$ we have that $\mathbb{E}_a[T_b] \ge \alpha_b- \alpha/4 \ge \alpha/4 $ (using the fact that $\alpha_x \ge \alpha/2$ for all $x \in D$). The proof is concluded as above using Matthews' method. \qed

\medskip

 We now prove Theorems \ref{prop:cover} and \ref{prop:cover2}.

 \textbf{Proof of Theorems \ref{prop:cover} and \ref{prop:cover2} without \eqref{e:nonconcenntration}.}
Recall that Aldous \cite{aldouscover} showed that in the reversible setup  $\sfrac{ \tau_{\mathrm{cov}}^{(n)}}{ t_{\mathrm{cov}}^{(n)}} \to 1 $ in distribution for all sequences of initial states iff  $\hit^{(n)} \ll t_{\mathrm{cov}}^{(n)} $. By \eqref{e:M>5} in the transitive setup, and by  \eqref{e:M>99} in the setup of Theorem \ref{prop:cover2},
  this occurs iff $\mathrm{gap}^{(n)}t_{\mathrm{cov}}^{(n)} \to \infty $. Indeed, the condition  $\mathrm{gap}^{(n)}t_{\mathrm{cov}}^{(n)} \to \infty $ is necessary for $\hit^{(n)} \ll t_{\mathrm{cov}}^{(n)} $ by \eqref{e:mixlehit}. Conversely, if   $\mathrm{gap}^{(n)}t_{\mathrm{cov}}^{(n)} \to \infty $ and along a subsequence $\hit^{(n)} \asymp t_{\mathrm{cov}}^{(n)} $, then along this subsequence $\mathrm{gap}^{(n)}\hit^{(n)} \to \infty $, which by \eqref{e:M>5} in the transitive setup and by \eqref{e:M>99} in the setup of Theorem \ref{prop:cover2}  implies that  $\hit^{(n)}\ll t_{\mathrm{cov}}^{(n)} $ along this subsequence. A contradiction! \qed

\subsection{Proof of \eqref{e:nonconcenntration} and stopping rules}
\label{s:tstop}

Before proving \eqref{e:nonconcenntration}, we first recall a notion of mixing,  first introduced by Aldous \cite{aldous1982some}
in the continuous-time setup, and later studied in discrete-time by Lov{\'a}sz
and Winkler \cite{LW}  who developed a rich theory  and also by Peres and Sousi \cite{PS} and by Oliveira \cite{Olivehit}:
\[t_{\mathrm{stop}}:=\max_x t_{\mathrm{stop},x}, \qquad \text{where} \] 
\begin{equation}
\label{eq:deftstop}
  t_{\mathrm{stop},x} := \inf \{\mathbb{E}_{x}[T]: T \text{
is a stopping rule
such that }\Pr_{x}[X_{T} \in \cdot]= \pi(\cdot)  \},
\end{equation}
and where a \emph{stopping rule} is a stopping
time, possibly with respect to a filtration larger than the natural filtration. A stopping rule attaining the infimum in \eqref{eq:deftstop}  is called \emph{mean optimal}.  Lov{\'a}sz
and Winkler  \cite{LW}  showed that for every initial state $x$ there  exists a mean optimal stopping rule $T$ and  \cite[Theorem 2.2]{LW}  that $T$ is mean optimal  iff there exists a state $y$ such that a.s.
\[T \le T_y. \]
Such a state is called a \emph{halting state}.
While they work in discrete-time, a standard application of Wald's equation can be used to translate their results to the continuous-time setup (cf.\ \cite{PS}; alternatively, one can simply check that the arguments in \cite{LW} can be carried out  directly in continuous-time). 

 Aldous \cite{aldous1982some} showed that under reversibility $\frac{1}{C} \le t_{\mathrm{stop}} / \mix^{\mathrm{TV}} \le C $ for some universal constant $C$.  This was refined by Peres and Sousi \cite{PS} and independently by Oliveira \cite{Olivehit} (see also \cite[Ch.\ 24]{LPW}) who in particular showed that for reversible Markov chains $\mix \asymp \max_{x,A:\, \pi(A) \ge \alpha}\E_x[T_A]$ for all fixed $\alpha> 1/2$ (this was extended also to $\alpha=1/2 $ in \cite{1/2}). For more connections between hitting times and mixing times we refer the reader to \cite{basu,hermontech,L2}.

\medskip

\textbf{Proof of \eqref{e:nonconcenntration}.} We suppress the dependence on $n$. Fix some $x$ and some mean optimal stopping rule $T$ (such that $\Pr_{x}[X_{T} \in \cdot]= \pi(\cdot)$). Let $y$ be  an halting state. Then by the strong Markov property, and the fact that $T \le T_y$ a.s., we have that for all $t$,
\[\Pr_x[T_y \ge t] \ge \Pr_x[T_y-T \ge t]=\Pr_{\pi}[T_y \ge t].  \] 
By \eqref{e:poscorfortail} 
\[\Pr_{\pi}[T_y \ge  i \alpha_y ] \ge (\Pr_{\pi}[T_y \ge   \alpha_y /2])^{2i} \ge \left( \frac{\alpha_y^2}{4 \mathbb{E}_{\pi}[T_y^2] } \right)^{2i} \] where we have used the Paley-Zygmund inequality in the second inequality. Using the assumption   $\min_z \alpha_z \asymp t_{\mathrm{cov}}$, we conclude the proof by arguing that $\alpha_y^2 \asymp  \mathbb{E}_{\pi}[T_y^2] $.  Indeed, by the aforementioned assumption $\alpha_y \asymp \hit  $, whereas  $\E_{\pi}[T_y^{2}] \le 2 \hit^2 $ by \eqref{e:exptailpi}.           
\qed

\subsection{Proof of \eqref{e:M>4} and \eqref{e:M>88}.}
\label{s:3}
To conclude the proof of Propositions \ref{p:transitivecover} and \ref{p:transitivecover2} it remains to prove \eqref{e:M>4} and \eqref{e:M>88}. 

\textbf{Proof of \eqref{e:M>4}.} Let $\eps >0 $. Let $\gd:=\eps/2$. Consider $s(\delta):=\sfrac{|\log \gd|}{\la_2} $. Let $x \in V$ and
\[A=A(x,\gd):= \{z:\int_0^{s(\delta)} H_{t}(z,x)dt \ge \sfrac{ \delta \alpha+s(\delta)}{|V|}   \}. \]
 By Markov's inequality, we have that 
  \begin{equation*}
 \label{e:piofA'}
\pi(A) \le \sfrac{\sum_{z \in V} \pi(z) \int_0^{s(\delta)} H_{t}(z,x)dt}{(\delta \alpha+s(\delta))/|V|} = \sfrac{s(\delta)}{s(\delta)+\delta \alpha}= \sfrac{|\log \delta |}{ |\log \delta |+\delta \la_2 \alpha}.
 \end{equation*}
By \eqref{e:Z} and \eqref{e:Z1} as well as by the definition of $A$, for all $z \notin A$ we have that
\[\alpha- \E_z[T_x]= \int_0^{s(\delta)}\left( H_{t}(z,x)|V|-1 \right)dt+\int_{s(\delta)}^{\infty}\left( H_{t}(z,x)|V|-1 \right)dt  \le \delta \alpha+\delta \alpha =\eps\alpha. \, \qed \]

\textbf{Proof of \eqref{e:M>88}.} Let $x \in V$. Let $\eps >0 $. Let $\gd=\gd(\eps):= \sfrac{ \eps \alpha }{2 \hit}$. Consider $s(\delta):=\sfrac{|\log \gd|}{\la_2} $ and
\[A=A(x,\gd):= \{z:\int_0^{s(\delta)} H_{t}(z,x)dt \ge \pi(x)(\sfrac{\eps}{2} \alpha+s(\delta))   \}. \]
 By Markov's inequality, we have that 
  \begin{equation*}
 \label{e:piofA''}
\pi(A) \le \sfrac{\sum_{z \in V} \pi(z) \int_0^{s(\delta)} H_{t}(z,x)dt}{(\sfrac{\eps}{2} \alpha+s(\delta))\pi(x) } = \sfrac{s(\delta)}{s(\delta)+\sfrac{\eps}{2}\alpha}= \sfrac{2|\log \delta |}{ 2|\log \delta |+ \eps \la_2 \alpha}. 
\end{equation*}
By \eqref{e:Z} and \eqref{e:Z1} as well as by the definitions of $A$ and $\delta$, for all $z \notin A$ we have that
\[ \alpha_{x}-\E_z[T_x]= \int_0^{s(\delta)}\left(\sfrac{ H_{t}(z,x)}{\pi(x)}-1 \right)dt+\int_{s(\delta)}^{\infty}\left( \sfrac{ H_{t}(z,x)}{\pi(x)}-1 \right)dt  \le \sfrac{\eps}{2} \alpha+ \half \delta (\alpha_{x}+\alpha_z) \le \eps \alpha.  \]
By \eqref{e:Z} we have that $\alpha_z-\E_x[T_z]=\alpha_{x}-\E_z[T_x] \le \eps \alpha $.   \qed

\section{Analysis of Example 1.1}
\label{s:EX3}
Let $G_n$ be an $n \times m $ grid torus with $m=m(n) \le n$.
It is well-known that for SRW on the $n$-cycle the spectral gap is $\asymp n^{-2} $ (cf.\ \cite[Example 12.10]{LPW}). Hence by general results about product chains (e.g.\ \cite[Corollary 12.13]{LPW})   $\mathrm{gap}(G_n) \asymp n^{-2} $ (uniformly for all $m(n) \le n$). 

We first consider the case that $m \in [ n / \log n,n] $. We now prove that $t_{\mathrm{cov}} \lesssim mn(\log n)^2 $  for such $m$. The same bound for $m \in[ \sfrac{n} {(\log n)^2}, \sfrac{n} {\log n}]  $ will be given at the end of the section. Let $H_t:=e^{-t(I-P)}$ be the heat kernel of the continuous-time SRW on $G_{n}$. Let $H_{t}^{(1,k)}$ be the heat kernel of  the continuous-time SRW on the $k$-cycle. Then  $\max_{a,b} H_t^{(1,k)}(a,b) \le \sfrac{C_0}{\sqrt{t+1}} $ for $t \le k^2$ (this follows by the local CLT, e.g.\ \cite[\S4.4]{frog}, or from some more general considerations, e.g.\ \cite[Theorem 17.17]{LPW}), while by the Poincar\' e inequality \[ \max_{a,b} |H_{t+k^2}^{(1,k)}(a,b)/k-1| \le \max_{a} |H_{k^2}^{(1,k)}(a,a)/k-1|e^{-ct/k^2}  \le C_{1}e^{-ct/k^2}\] (for some absolute constants $c,C_0,C_1$) for all $t \ge 0$.     Because the continuous-time chain evolves independently in the two coordinates, at rate $\half$ along each coordinate (e.g.\ \cite[p.\ 288 (20.35)]{LPW}) for all $a=(a_1,a_2)$ and $b=(b_1,b_2) $ 
\begin{equation}
\label{e:factor}
H_t(a,b)=H_{t/2}^{(1,n)}(a_{1},b_{1})H_{t/2}^{(1,m)}(a_{2},b_{2}).
\end{equation} Denote the vertex set of $G_n$ by $V_n$ and the uniform distribution on $V_n$ by $\pi$.  It follows that for all $y \in V_n $ we have that $  \int_{2n^{2}}^{\infty}\left(H_{t}(y,y)-\pi(y)\right) dt \le C_2  \int_{0}^{2n^{2}}\left(H_{t}(y,y)-\pi(y)\right) dt    $ and so
 \begin{equation}
 \label{e:Htyy0inf} 
 \begin{split}
  \int_{0}^{\infty}\left(H_{t}(y,y)-\pi(y)\right) dt & \le C_{2}  \int_{0}^{2m^2}H_{t}(y,y) dt +C_2 \int_{2m^{2}}^{2n^{2}}H_{t}(y,y) dt  
\\ & \le C_3 \int_{0}^{2m^2} \frac{dt}{1+t}  +C_4\int_{2m^2}^{2n^2} \frac{dt}{m \sqrt{ t/2}}\le C_5(\sfrac{n}{m}+ \log m).  
\end{split}
\end{equation}
Hence by Fact \ref{f:a} and \eqref{e:alphax}  \[\hit(G_n) \le 2 \alpha(G_n) \le 2C_5(n^2+nm  \log m) ,\] and so by \eqref{e:mixlehit} \[t_{\mathrm{cov}}(G_n) \le C_6 (n^2+nm  \log m) \log n .\] For $m \in [ n/ \log n,n] $ we get that $t_{\mathrm{cov}}(G_n) \lesssim nm  (\log n)^2 $.

 We now prove a matching lower bound for $m \in [ n/ (\log n)^{2},n] $. 
If $x,y \in V_n $ are of graph distance at least $\sqrt{m} $ then by the local CLT (e.g.\ \cite[\S4.4]{frog})  
\[\forall \, t \ge 0, \qquad H_{t}(y,y)-H_{t}(x,y) \ge c_1 (t+1)^{-1}\Ind{t \le 2m },\]
where we have used the fact that for transitive chains $H_{t}(y,y) \ge H_{t}(x,y) $ for all $x,y$ and all $t$ (e.g.\ \cite[Eq.\ (3.60)]{aldousfill}). Hence by \eqref{e:ZZZ} for such $x,y$ we have that  
 \begin{equation}
 \label{e:Htyy0inf'} 
 \begin{split}
 \E_y[T_x] =\E_x[T_y] =mn \int_{0}^{\infty}\left(H_{t}(y,y)-H_{t}(x,y)\right) dt  \ge  
  c_2 mn \log m,  
\end{split}
\end{equation}

By considering a collection of vertices $A \subset V_n $ of size $\Omega( n ) $ such that for any distinct $a,b \in A$ we have that the distance of $a$ from $b$ is at least $\sqrt{m}$ we get by \eqref{e:Htyy0inf'} and Matthews' argument that  $t_{\mathrm{cov}}(G_n) \gtrsim nm  (\log n)^2 $ (for $m \in [ n/ (\log n)^{2},n] $).  

We now treat  $m \in[1, n/ (\log n)^2]  $. Using our upper bound on the spectral-gap, it remains only need to show that  $t_{\mathrm{cov}}(G_n) \lesssim n^{2}$ in this regime. This requires a more careful analysis than the one above. 

It is not hard to show that there exists an absolute constant $p \in (0,1) $ such that for every $n$ and every $C \ge 1$,  SRW on the $n$-cycle satisfies that  all vertices are visited at least $Cn /4$ times by time $Cn^2$ with probability at least $p$. For an argument which is specific for the cycle cf.\ \cite[Lemma 6.6]{BH}. This also follows from a general result from \cite{DLP} which says that the blanket-time (see \cite{DLP} for a definition) is proportional to the cover time. Thus with probability at least $p$, by time $8C n^2$ for all $i \in [n]$ the walk spends at least $Cn$ steps at each strip $S_{i}:=\{i\} \times [m] $,  where $[k]:=\{1,\ldots,k \}$. We now exploit this fact to obtain the bound   $t_{\mathrm{cov}}(G_n) \lesssim n^{2}$.

For a set $A$ we define the \emph{induced chain} on $A$, denoted by $(Y_k^A)_{k=0}^{\infty}$, to be the chain $(X_t)_{t \ge 0}$ viewed only at times at which it visits $A$. That is $t_0(A):=\inf\{t \ge 0 : X_t \in A \} $, $Y_{0}^{A}:=X_{t_0(A)}$ and inductively,  \[t_{k+1}(A):=\inf\{t > t_k (A): X_t \in A, X_s \neq X_{t_k (A)} \text{ for some }s \in (t_k (A),t] \} \] and $Y_{k+1}^A:=X_{t_{k+1}(A)}$. Observe that if   $\tau_{\mathrm{cov}}^{A}:=\inf\{t:\{X_s :s \le t \} \supseteq A  \} $  and  $\tau_{\mathrm{cov}}^{\mathrm{induced},A}:=\inf\{k:\{Y_i^{A} :i \le k \} = A  \} $ are the cover times of $A$ w.r.t.\ the original chain and the induced chain on $A$, respectively, and $\Pr^{\mathrm{induced},A}$ and $\E^{\mathrm{induced},A}$ are the law and expectation of the induced chain on $A$, then for any partition $A_1,\ldots,A_\ell$ of $V$ we have that for all $v \in V$
\[\forall \, t \ge 0, k \in \N, \quad \Pr_v[\tau_{\mathrm{cov}}> t ] \le \Pr_v[\max_{i \in[\ell] } t_{k}(A_{i})> t ]+\sum_{j \in[\ell]} \max_{a \in A_{j} }\Pr_a^{\mathrm{induced},A_{j}}[\tau_{\mathrm{cov}}^{\mathrm{induced},A_{j}}>k], \]
where we have used the fact that for all $j \in [\ell]$ and $k \in \N$ we have that \[\Pr_v[\tau_{\mathrm{cov}}^{A_{j}}>t,\max_{i  \in[\ell] } t_{k}(A_{i})\le t ] \le \Pr_v[ \{X_{t_{i}(A_{j})}^{A_{j}} :i \le k \} \neq A_j ]= \max_{a \in A_{j} }\Pr_a^{\mathrm{induced},A_{j}}[\tau_{\mathrm{cov}}^{\mathrm{induced},A_{j}}>k]. \]
We  now argue that it suffices to show  that the induced chain  on each  strip $S_i$ satisfies that the probability that it is not covered in  $Cn$ steps is $\ll 1/n$, provided that  $C$ is sufficiently large (note that here steps are counted w.r.t.\ the induced chain, i.e.\ these are the number of visits to each strip). That is, by symmetry, it suffices to show that $\max_{a \in S_{1} }\Pr_a^{\mathrm{induced},S_{1}}[\tau_{\mathrm{cov}}^{\mathrm{induced},S_{1}}>\lceil Cn \rceil]=o(1/n) $. Indeed, once this is established for some $C \ge 1 $ then using the above with the partition $S_1,\ldots,S_n$ we get that 
\[\Pr_v[\tau_{\mathrm{cov}}> 8Cn^{2} ] \le \Pr_v[\max_{i \in [n] } t_{\lceil Cn \rceil }(S_{i})> 8Cn^{2}  ]+n \max_{a \in S_{1} }\Pr_a^{\mathrm{induced},S_{1}}[\tau_{\mathrm{cov}}^{\mathrm{induced},S_{1}}>\lceil Cn \rceil].    \] 
The r.h.s.\ is at most $1-p+o(1) $ where $p \in (0,1) $ is as above. Using the obvious submultiplicativity property\[\forall \, s \ge 0, r \in \N, \qquad \max_v \Pr_v[\tau_{\mathrm{cov}}> sr] \le (\max_v \Pr_v[\tau_{\mathrm{cov}}> s])^{r},   \]
this implies that $t_{\mathrm{cov}}(G_n) \lesssim n^2  $ as desired.

Denote
$M:=\max_{x,y \in S_1 } \mathbb{E}_x^{\mathrm{induced},S_1}[T_y]$. We will show that $M \le C_{6} m \log m$. By Markov's inequality $\Pr_a^{\mathrm{induced}}[T_b>2M] \le \half $ for all $a,b$ in the same strip.  By the Markov property, for all $a,b$ in the same strip, we have that
\[\Pr_a^{\mathrm{induced}}[T_b>(2M) \lceil \log_2(n^3)\rceil ] \le 2^{-\log_2(n^3)  } = n^{-3}. \]
By a union bound over all $m$ vertices in that strip we obtain the desired tail estimate on the cover time of a single strip w.r.t.\ the induced chain, as $(2M) \lceil \log_2(n^3)\rceil \lesssim   m \log m \log n \lesssim  n $. 

It remains to show that $M=O(m \log m) $.
  The induced chain is itself a transitive chain.  In particular, its stationary distribution is the uniform distribution.  Let $x,y \in S_1$.  By Wald's equation, and the fact that the expected return time to $S_1$ from any $a \in S_1$ is $n$, we have that $\E_x[T_y]=n \mathbb{E}_x^{\mathrm{induced},S_1}[T_y] $, and so by \eqref{e:ZZZ} we have that  
 \begin{equation}
 \label{e:eind}
 \mathbb{E}_x^{\mathrm{induced},S_1}[T_y]=m \int_{0}^{\infty}[H_{t}(y,y)-H_{t}(x,y)]dt . \end{equation}

Finally, for $x=(1,x_{2}),y=(1,y_{2}) \in S_1$ by the local CLT (e.g.\ \cite[\S4.4]{frog}) and \eqref{e:factor}
\begin{equation}
\begin{split}
\label{e:6.5}
&H_{t}(y,y)-H_{t}(x,y)\leq H_{t/2}^{(1,n)}(1,1) \max_{a,b} \left( H_{t/2}^{(1,m)}(1,a)-H_{t/2}^{(1,m)}(1,b)\right)  \\ & \lesssim (t+1)^{-1}\Ind{t \le 2m^{2} }+\frac{1}{m \sqrt{t}} e^{-ct/m^{2}}\Ind{2m^{2} < t \le 2n^2 }+\frac{1}{m n} e^{-ct/m^{2}}\Ind{  t > 2n^2 },
\end{split}
\end{equation}
where we have used the Poincar\'e inequality for SRW on the $m$-cycle. Integrating we see that for $m \in [n/\log n,n] $ indeed 
\begin{equation}
\label{e:6.3} \max_{x,y \in S_1 } \int_{0}^{\infty}[H_{t}(y,y)-H_{t}(x,y)]dt \lesssim \log m   .
\end{equation}
 By \eqref{e:eind} we get that indeed $M \lesssim m \log m  $, as desired. 

We now show that also for $m \in [n / (\log n)^2,n/\log n] $ we have $t_{\mathrm{cov}} \lesssim nm (\log n)^2 $. Indeed by \eqref{e:6.3} $M \lesssim m \log m  $ and so by  the above analysis we can bound $t_{\mathrm{cov}}$ from above, up to a constant factor, by the expected time until all strips are visited for at least $CM \log n $ time units. As $M \log n \gtrsim n $ it is not hard to to show that this time is of order $n \times (CM \log n) \asymp nm (\log n)^2$ by using the fact that there exists $C>0$ such that with probability at least $p>0$ all strips are visited for at least $n$ time units by time $Cn^{2}$. To see this, consider a sequence of consecutive time intervals of length $Cn^{2}$, and use the Markov property to argue that during each interval, with probability at least $p>0$, independently of the previous time intervals, all strips are visited for at least $n$ time units). \qed 

\appendix

{\bf Acknowledgements:} The author would like to thank  Perla Sousi for carefully reading a previous draft of this work and finding several typos. This  paper  greatly  benefited  from  discussions  with  Ofer Zeitouni  about the cover time of an $n \times m$ grid torus and with Matthew Tointon about the Open problems from \S\ref{s:open}. We also thank the anonymous referee for thoroughly reading the manuscript and for finding a mistake in the analysis of Example 1.1 in a previous draft, as well as multiple typos.


\begin{thebibliography}{BKC}


\bibitem{aldoushitting}
Aldous, D., Hitting times for random walks on vertex-transitive graphs. \emph{Math. Proc. Cambridge Philos. Soc.}  \textbf{106} (1989), no. 1, 179--191. \href{http://www.ams.org/mathscinet-getitem?mr=MR0994089}{\textcolor{blue}{MR0994089}}.

\bibitem{aldousr}
Aldous, D.,
\newblock Random walk covering of some special trees.
\newblock \emph{J. Math. Anal. Appl.} 157 (1991), no. 1, 271--283. \href{http://www.ams.org/mathscinet-getitem?mr=MR1109456}{\textcolor{blue}{MR1109456}}

\bibitem{aldous1982some}
Aldous, D.,
\newblock Some inequalities for reversible Markov chains.
\newblock {\em Journal of the London Mathematical Society}, 2(3):564--576,
  1982. \href{http://www.ams.org/mathscinet-getitem?mr=MR657512}{\textcolor{blue}{MR657512}}

\bibitem{aldouscover}
Aldous, D., Threshold limits for cover times. \emph{J. Theoret. Probab.} \textbf{4} (1991), no. 1, 197--211. \href{http://www.ams.org/mathscinet-getitem?mr=MR1088401}{\textcolor{blue}{MR1088401}}.



\bibitem{aldousfill}
Aldous, D., and Fill, J.,
  \emph{Reversible Markov chains and random walks on graphs}. Unfinished manuscript. Available at the first author's \href{https://www.stat.berkeley.edu/~aldous/}{\textcolor{blue}{website}}.





\bibitem{Bar}
Barlow, M., T., Ding, J., Nachmias, A., and Peres, Y.,
\newblock The evolution of the cover time.
\newblock \emph{Combin. Probab. Comput.} 20 (2011), no. 3, 331--345. \href{http://www.ams.org/mathscinet-getitem?mr=MR2784631}{\textcolor{blue}{MR2784631}}

\bibitem{basu}
Basu, R., Hermon, J., Peres, Y., Characterization of cutoff for reversible Markov chains. \emph{Ann. Probab.} 45 (2017), no. 3, 1448--1487. \href{http://www.ams.org/mathscinet-getitem?mr=MR3650406}{\textcolor{blue}{MR3650406}}


\bibitem{Bel}
Belius, D.,
\newblock Gumbel fluctuations for cover times in the discrete torus. 
\newblock \emph{Probab. Theory Related Fields} 157 (2013), no. 3-4, 635--689. \href{http://www.ams.org/mathscinet-getitem?mr=MR3129800}{\textcolor{blue}{MR3129800}}

\bibitem{BK}
Belius, D., and Kistler, N., 
\newblock The subleading order of two dimensional cover times. 
\newblock \emph{Probab. Theory Related Fields} 167 (2017), no. 1-2, 461--552. \href{http://www.ams.org/mathscinet-getitem?mr=MR3602852}{\textcolor{blue}{MR3602852}} 

\bibitem{BH} Benjamini, I., and Hermon, J., 
\newblock Recurrence of Markov chain traces. To appear in \emph{Annales Henri Poincar\'e prob. and statistics.}
\newblock (2019) Arxiv preprint \href{https://arxiv.org/abs/1711.03479}{\textcolor{blue}{arXiv:1711.03479}}

\bibitem{frog}
Benjamini, I.,  Fontes, L.\ R.,  Hermon, J., and Machado, F.\ P.,
\newblock On an epidemic model on finite graphs. To appear in \emph{Ann. Appl. Probab.}
\newblock (2019) Arxiv preprint \href{https://arxiv.org/abs/1610.04301}{\textcolor{blue}{arXiv:1610.04301}}


\bibitem{BH2} Benjamini, I., Hermon, J., Tessera, R., and Tointon, M.,  
\newblock Transitive graphs with proportional mixing and hitting times scale to the unit-circle. In preparation (2019+) 

\bibitem{BT}
Breuillard, E., and Tointon, M.\ C.\ H., Nilprogressions and groups with moderate growth. \emph{Adv. Math.} 289 (2016), 1008--1055. \href{http://www.ams.org/mathscinet-getitem?mr=MR3439705}{\textcolor{blue}{MR3439705}}





\bibitem{Chen}
Chen, G.\ Y., and Saloff-Coste, L., The cutoff phenomenon for ergodic Markov processes.
 \newblock \emph{Electron. J. Probab.} 13 (2008), no. 3, 26--78.  
\href{http://www.ams.org/mathscinet-getitem?mr=MR2375599}{\textcolor{blue}{MR2375599}}
\bibitem{CP}
Comets, F., Gallesco, C., Popov, S, and Vachkovskaia, M.,
\newblock On large deviations for the cover time of two-dimensional torus. \newblock \emph{Electron. J. Probab.} 18 (2013), no. 96, 18 pp. 
\href{http://www.ams.org/mathscinet-getitem?mr=MR3126579}{\textcolor{blue}{MR3126579}}

\bibitem{CF}
Cooper, C., and Frieze, A., 
\newblock The cover time of the giant component of a random graph.
\newblock \emph{Random Structures Algorithms} 32 (2008), no. 4, 401--439. \href{http://www.ams.org/mathscinet-getitem?mr=MR2422388}{\textcolor{blue}{MR2422388}}

\bibitem{DSC}
Diaconis, P. and Saloff-Coste, L., Moderate growth and random walk on finite groups. \emph{Geom. Funct. Anal.} 4 (1994), no. 1, 1--36. \href{http://www.ams.org/mathscinet-getitem?mr=MR1254308}{\textcolor{blue}{MR1254308}}

\bibitem{DPRZ}
Dembo, A., Peres, Y., Rosen, J., and Zeitouni, O., 
\newblock Cover times for Brownian motion and random walks in two dimensions. \newblock \emph{Ann. of Math.} (2) 160 (2004), no. 2, 433--464. \href{http://www.ams.org/mathscinet-getitem?mr=MR2123929}{\textcolor{blue}{MR2123929}}

\bibitem{DS}
Diaconis, P., and Saloff-Coste, L. 
\newblock Walks on generating sets of groups.
\newblock \emph{Invent. Math.} 134 (1998), no. 2, 251--299. \href{http://www.ams.org/mathscinet-getitem?mr=MR1650316}{\textcolor{blue}{MR1650316}} 
\bibitem{Ding}
Ding, J.,
\newblock Asymptotics of cover times via Gaussian free fields: bounded-degree graphs and general trees.
\newblock \emph{Ann. Probab.} 42 (2014), no. 2, 464--496. \href{http://www.ams.org/mathscinet-getitem?mr=MR3178464}{\textcolor{blue}{MR3178464}}

\bibitem{Ding2}
Ding, J.,
\newblock On cover times for 2D lattices.
\newblock \emph{Electron. J. Probab.} 17 (2012), no. 45, 18 pp. \href{http://www.ams.org/mathscinet-getitem?mr=MR2946152}{\textcolor{blue}{MR2946152}} 

\bibitem{DLP}
Ding, J., Lee, J.\ R., and Peres, Y. 
\newblock Cover times, blanket times, and majorizing measures. 
\newblock \emph{Ann. of Math.} (2) 175 (2012), no. 3, 1409--1471. \href{http://www.ams.org/mathscinet-getitem?mr=MR2912708}{\textcolor{blue}{MR2912708}} 
\bibitem{DZ}
Ding, J., and Zeitouni, O.,
\newblock A sharp estimate for cover times on binary trees.
\newblock \emph{Stochastic Process. Appl.} 122 (2012), no. 5, 2117--2133. \href{http://www.ams.org/mathscinet-getitem?mr=MR2921974}{\textcolor{blue}{MR2921974}}

\bibitem{Feige1}
Feige, U.,
\newblock A tight lower bound on the cover time for random walks on graphs. \newblock \emph{Random Structures Algorithms} 6 (1995), no. 4, 433--438. \href{http://www.ams.org/mathscinet-getitem?mr=MR1368844}{\textcolor{blue}{MR1368844}}

\bibitem{Feige2}
Feige, U.,
\newblock A tight upper bound on the cover time for random walks on graphs. \newblock \emph{Random Structures Algorithms} 6 (1995), no. 1, 51--54. \href{http://www.ams.org/mathscinet-getitem?mr=MR1368834}{\textcolor{blue}{MR1368834}}

\bibitem{hermontech}
Hermon, J.,
\newblock A technical report on hitting times, mixing and cutoff.
\newblock \emph{ALEA Lat. Am. J. Probab. Math. Stat.} 15 (2018), no. 1, 101--120. \href{http://www.ams.org/mathscinet-getitem?mr=MR3765366}{\textcolor{blue}{MR3765366}}

\bibitem{hermoninter}
Hermon, J.,
\newblock Intersection times for critical branching random walk. (2019),
\newblock Arxiv preprint \href{https://arxiv.org/abs/1908.08525}{\textcolor{blue}{arXiv:1908.08525}}

\bibitem{HLP}
Hermon, J., Lacoin, H., and Peres, Y.,
\newblock Total variation and separation cutoffs are not
equivalent and neither one implies the other. \emph{Electron. J. Probab.} 21 (2016), no. 44, 36 pp. \href{http://www.ams.org/mathscinet-getitem?mr=MR3530321}{\textcolor{blue}{MR3530321}}

\bibitem{L2}
Hermon, J., and Peres, Y.,
\newblock A characterization of $L_2$ mixing and hypercontractivity via hitting times and maximal inequalities.
\newblock \emph{Probab. Theory Related Fields} 170 (2018), no. 3-4, 769--800. \href{http://www.ams.org/mathscinet-getitem?mr=MR3773799}{\textcolor{blue}{MR3773799}}

%
%
%
%
%
%
%
%


\bibitem{1/2}
Griffiths, S., Kang, R., J., Oliveira, R., Patel, V.,
\newblock Tight inequalities among set hitting times in Markov chains.
\newblock \emph{Proc. Amer. Math. Soc.} 142 (2014), no. 9, 3285--3298. 
\href{http://www.ams.org/mathscinet-getitem?mr=MR3223383}{\textcolor{blue}{MR3223383}} 









\bibitem{Kahn}
Kahn, J., Kim, J.\ H., Lov\' asz, L., and Vu, V.\ H.,
\newblock The cover time, the blanket time, and the Matthews bound.
\newblock 41st Annual Symposium on Foundations of Computer Science, 467--475, IEEE Comput. Soc. Press, Los Alamitos, CA, 2000. \href{http://www.ams.org/mathscinet-getitem?mr=MR1931843}{\textcolor{blue}{MR1931843}} 

\bibitem{LPW} Levin, D., and Peres, Y., (2017).
 \newblock {\em Markov chains and mixing times.} 
 American Mathematical Society, Providence, RI. With contributions by Elizabeth L.\ Wilmer and a chapter by James G.\ Propp and David B.\ Wilson. \href{http://www.ams.org/mathscinet-getitem?mr=MR3726904}{\textcolor{blue}{MR3726904}}

\bibitem{LW}
Lov{\'a}sz, L., and Winkler, P.,
\newblock Mixing times.
\newblock {\em Microsurveys in discrete probability} 41:85--134, 1998. \href{http://www.ams.org/mathscinet-getitem?mr=MR1630411}{\textcolor{blue}{MR1630411}} 
\bibitem{LP}
Lyons, R., and Peres, Y., \emph{Probability on trees and networks}. Cambridge Series in Statistical and Probabilistic Mathematics, 42. Cambridge University Press, New York, 2016. \href{http://www.ams.org/mathscinet-getitem?mr=MR3616205}{\textcolor{blue}{MR3616205}} 

\bibitem{Matthews}
Matthews, P., 
\newblock Covering problems for {M}arkov chains.
\newblock \emph{Ann. Probab.} 16 (1988), no. 3, 1215--1228. \href{http://www.ams.org/mathscinet-getitem?mr=MR0942764}{\textcolor{blue}{MR0942764}}

\bibitem{MP}
Miller, J., and Peres, Y,
\newblock Uniformity of the uncovered set of random walk and cutoff for lamplighter chains.
\newblock \emph{Ann. Probab.} 40 (2012), no. 2, 535--577. \href{http://www.ams.org/mathscinet-getitem?mr=MR2952084}{\textcolor{blue}{MR2952084}}


\bibitem{Olivehit}
Oliveira, R.,
\newblock  Mixing and hitting times for finite Markov chains.
\newblock \emph{Electron. J. Probab.} 17 (2012), no. 70, 12 pp. \href{http://www.ams.org/mathscinet-getitem?mr=MR2968677}{\textcolor{blue}{MR2968677}}



\bibitem{PS}
Peres, Y., and Sousi, P.,
\newblock Mixing times are hitting times of large sets.
\newblock \emph{J. Theoret. Probab.} 28 (2015), no. 2, 488--519. \href{http://www.ams.org/mathscinet-getitem?mr=MR0942764}{\textcolor{blue}{MR3370663}}

\bibitem{TT} Tessera, R., and Tointon, M.,
\newblock A finitary structure theorem for vertex-transitive graphs of polynomial growth. (2019),
 \newblock Arxiv preprint \href{https://arxiv.org/abs/1908.06044}{\textcolor{blue}{arXiv:1908.06044}}

\bibitem{TT2} Tessera, R., and Tointon, M., \newblock Sharp relations between volume growth, isoperimetry and resistance in vertex-transitive graphs. \newblock In preparation (2019+)

\bibitem{Zhai}
Zhai, A., 
\newblock Exponential concentration of cover times.
\newblock \emph{Electron. J. Probab.} 23 (2018), Paper No. 32, 22 pp. \href{http://www.ams.org/mathscinet-getitem?mr=MR3785402}{\textcolor{blue}{MR3785402}}


\end{thebibliography}
\end{document}